\documentclass[12pt,a4paper]{article}

\usepackage{amsmath,amssymb,amsthm,graphicx,fixmath}

\newcommand*{\calF}{\mathcal{F}}%
\newcommand*{\oset}[1]{\left<#1\right>}
\newcommand*{\ceil}[1]{\lceil#1\rceil}%
\newcommand*{\floor}[1]{\lfloor#1\rfloor}%
\newcommand*{\Positives}{\ensuremath{\mathbb{Z}^+}}%
\newcommand*{\chicf}{\chi_{\text{\textrm{\textup{cf}}}}}%
\newcommand*{\chiodd}{\chi_{\text{\textrm{\textup{odd}}}}}%
\newcommand*{\chium}{\chi_{\text{\textrm{\textup{um}}}}}%
\newcommand*{\chirb}{\chi_{\text{\textrm{\textup{rb}}}}}%
\newcommand*{\chiump}{\chium^{\text{\textrm{\textup{p}}}}}%
\newcommand*{\chicfp}{\chicf^{\text{\textrm{\textup{p}}}}}%
\newcommand*{\chioddp}{\chiodd^{\text{\textrm{\textup{p}}}}}%
\newcommand*{\gensymbol}{{\diamond}}%
\newcommand*{\chigenp}{\chi_{\gensymbol}^{\text{\textrm{\textup{p}}}}}
\newcommand*{\containssubgraph}{\supseteq}%
\newcommand*{\subgraph}{\subseteq}%
\newcommand*{\issubgraphof}{\subgraph}%
\newcommand*{\issupergraphof}{\containssubgraph}%
\newcommand*{\card}[1]{\lvert#1\rvert}%
\newcommand*{\fracol}[2]{{#1}/{#2}}%
\newcommand*{\issubsetof}{\subseteq}%

\newcommand*{\ODD}{\text{\textrm{\textup{ODD}}}}%
\newcommand*{\CF}{\text{\textrm{\textup{CF}}}}%
\newcommand*{\UM}{\text{\textrm{\textup{UM}}}}%

\theoremstyle{plain}
\newtheorem{theorem}{Theorem}[section]
\newtheorem{lemma}[theorem]{Lemma}

\newtheorem{clai}[theorem]{Claim}
\newtheorem{corollary}[theorem]{Corollary}

\theoremstyle{definition}
\newtheorem{defi}[theorem]{Definition}
\newtheorem{obs}[theorem]{Observation}
\newtheorem{example}[theorem]{Example}
\theoremstyle{remark}
\newtheorem{remark}[theorem]{Remark}

\begin{document}

\title
  {Unique-maximum and conflict-free coloring\\
   for hypergraphs and tree graphs%
  \thanks{A preliminary version of this work appeared in the
          38th International Conference on Current Trends
          in Theory and Practice of Computer Science 
          (SOFSEM 2012).}}

\date{}%

\author{%
        {Panagiotis Cheilaris}%
        \thanks{Department of Informatics,
                Universit{\`a} della Svizzera italiana,
                Lugano, Switzerland.}
   \and {Bal{\'a}zs Keszegh}%
        \thanks{Alfr\'ed R\'enyi Institute of Mathematics, 
          Budapest, Hungary.}
   \and {D{\"o}m{\"o}t{\"o}r P{\'a}lv{\"o}lgyi}%
        \thanks{E\"otv\"os University, Budapest, Hungary.}%
}


%

\maketitle

\begin{abstract}
We investigate the relationship 
between two kinds of vertex colorings of hypergraphs:
unique-maximum colorings and 
conflict-free colorings.
In a unique-maximum coloring, the colors are ordered, and 
in every hyperedge of the hypergraph the
maximum color in the hyperedge occurs in only one vertex of 
the hyperedge.
In a conflict-free coloring, 
in every hyperedge of the hypergraph there exists
a color in the hyperedge that occurs 
in only one vertex of the hyperedge.
We define corresponding unique-maximum and conflict-free chromatic
numbers and investigate their relationship in arbitrary
hypergraphs. Then, we concentrate on hypergraphs that are induced
by simple paths in tree graphs.
\end{abstract}



	   

\section{Introduction}

A hypergraph $H$ is a pair $(V,E)$, where $E$ (the hyperedge set)
is a family of non-empty subsets of $V$ (the vertex set).
A vertex coloring of a
hypergraph $H=(V,E)$ is a function $C \colon V \to \Positives$.

A hypergraph is a generalization of a graph.
Therefore, it is
natural to consider how to generalize proper vertex
coloring of a graph 
to a vertex coloring of a hypergraph. 
(In a proper vertex coloring of a graph, 
any two vertices neighboring with an
edge in the graph have to be assigned different colors by the
coloring function $C$.) 
Vertex coloring in hypergraphs can be defined in many ways, so that restricting
the definition to simple graphs coincides with proper graph coloring.

At one extreme, it is only required that the vertices of each hyperedge
are not all colored with the same color (except for singleton hyperedges).
This is called a \emph{non-monochromatic} coloring of a hypergraph. The
minimum number of colors necessary to color in such a way a
hypergraph $H$ is the (non-monochromatic) chromatic
number of $H$, denoted by $\chi(H)$.

At the other extreme, we can require that the vertices of each hyperedge
are all colored with different colors. This is called a
\emph{colorful} or \emph{rainbow} coloring of $H$ and we have the
corresponding rainbow chromatic number of $H$, denoted by
$\chirb(H)$.

In this paper we investigate the following two types of vertex colorings of
hypergraphs that are between the above two extremes.

\begin{defi} \label{def:ordcol}
A \emph{unique-maximum coloring}
of $H=(V,E)$ with $k$ colors
is a function $C \colon V \to \{1,\dots,k\}$
such that for  each $e \in E$
the \emph{maximum color in $e$ occurs exactly once} on the vertices of $e$.
The minimum $k$ for which a hypergraph $H$ has a unique-maximum coloring
with $k$ colors
is called the \emph{unique-maximum chromatic number} of $H$ and is
denoted by $\chium(H)$.
\end{defi}

\begin{defi} \label{def:cfcol}
A \emph{conflict-free coloring}
of $H=(V,E)$ with $k$ colors
is a function $C \colon V \to \{1,\dots,k\}$
such that for each $e \in E$ 
there is \emph{a color in $e$ that occurs exactly once} on the vertices of $e$.
The minimum $k$ for which a hypergraph $H$ has a conflict-free coloring
with $k$ colors
is called the \emph{conflict-free chromatic number} of $H$ and is
denoted by $\chicf(H)$.
\end{defi}

We also introduce a new coloring, that proves useful in showing 
lower bounds, and that could be of independent interest.

\begin{defi} \label{def:oddcfcol}
An \emph{odd coloring}
of $H=(V,E)$ with $k$ colors
is a function $C \colon V \to \{1,\dots,k\}$
such that for each $e \in E$ 
there is \emph{a color that occurs an odd number of times} on the vertices of $e$.
The minimum $k$ for which a hypergraph $H$ has an odd coloring
with $k$ colors
is called the \emph{odd chromatic number} of $H$ and is
denoted by $\chiodd(H)$.
\end{defi}

Every rainbow coloring is unique-maximum, every unique-maximum
coloring is conflict-free, and every conflict-free coloring is
odd and non-monochromatic. Therefore, for every hypergraph $H$,
$\max(\chi(H), \chiodd(H)) \leq \chicf(H) \leq \chium(H) \leq \chirb(H)$.
Note that an odd coloring can be monochromatic.

The study of 
conflict-free coloring hypergraphs started in \cite{ELRS03jo,Sm03},
with an emphasis in hypergraphs induced by geometric shapes.
The main application of conflict-free coloring is that it 
models a frequency assignment for cellular
networks. A cellular network consists of two kinds of nodes:
\emph{base stations} and \emph{mobile agents}. Base stations have
fixed positions and provide the backbone of the network; they are
represented by vertices in \(V\). Mobile agents are the clients of the
network and they are served by base stations. This is done as
follows: Every base station has a fixed frequency; this is represented 
by the coloring \(C\), i.e., colors represent frequencies. If an
agent wants to establish a link with a base station it has to tune
itself to this base station's frequency. Since agents are mobile,
they can be in the range of many different base stations. To avoid
interference, the system must assign frequencies to base stations
in the following way: For any range, there must be a base station
in the range with a frequency that is not used by some other
base station in the range. One can solve the problem by assigning
\(n\) different frequencies to the \(n\) base stations. However,
using many frequencies is expensive, and therefore, a scheme that
reuses frequencies, where possible, is preferable. Conflict-free coloring
problems have been the subject of many recent papers due to their
practical and theoretical interest (see e.g.\
\cite{PT03,HS03,Chenetal2006,EM05,%
BCS2008talg}).

Most approaches in the conflict-free coloring literature
rely on the stronger unique-maximum colorings (a notable exception is 
the `triples' algorithm in \cite{BCS2008talg}), 
because unique-maximum colorings are
easier to argue about in proofs, due to their additional
structure. Another advantage of unique-maximum colorings 
is the simplicity of computing the unique color
in any range (it is always the maximum color), given a
unique-maximum coloring, which can be helpful if very simple
mobile devices are used by the agents. 

Other hypergraphs that have been studied with respect to these
colorings, are ones which are induced by a graph and (a)~its
neighborhoods or (b)~its paths:
\begin{itemize}
\item[(a)]
Given a graph $G$, consider the hypergraph with the
same vertex set as $G$ 
and a hyperedge for every distinct vertex neighborhood of $G$; 
such conflict-free colorings have been studied in
\cite{CheilarisCUNYthesis2009,PachTardoscfgh2009}.
\item[(b)]
Given a graph $G$, consider the hypergraph $H$ with the
same vertex set as $G$ 
and a hyperedge for every distinct vertex set 
that can be spanned by a \emph{simple} path of $G$.
A unique-maximum (respectively conflict-free, odd) coloring of $H$ is called a
unique-maximum (respectively conflict-free, odd) coloring of $G$ with respect to
paths; we also define the corresponding graph chromatic numbers,
$\chiump(G) = \chium(H)$, $\chicfp(G) = \chicf(H)$ and $\chioddp(G) = \chiodd(H)$. 
Sometimes to improve readability of the text, we simply talk about the UM 
(respectively CF, ODD) chromatic number of a graph.
\end{itemize}

Unique-maximum colorings with respect to paths of graphs 
are known alternatively in the literature as
\emph{ordered colorings} or \emph{vertex rankings},
and the unique-maximum chromatic number is also known as
\emph{tree-depth} \cite{NMejc2006treedepth}.
The problem of computing such unique-maximum colorings is a well-known and
widely studied problem (see e.g.\ \cite{orderedcoloring}) with many
applications including VLSI design
\cite{DBLP:conf/focs/Leiserson80} and parallel Cholesky
factorization of matrices \cite{liu:134}. The 
problem is also interesting for the Operations Research community,
because it has applications in planning efficient assembly of
products in manufacturing systems \cite{optnoderanktree}. 
In general, it seems that the vertex ranking problem can model
situations where interrelated tasks have to be accomplished fast
in parallel (assembly from parts, parallel query optimization in
databases, etc.).
For general graphs, finding the exact unique-maximum chromatic
number with respect to paths of a
graph is NP-complete 
\cite{Pothen1988TR,LTT1989dam,rankings_of_graphs_1998,NMejc2006treedepth} 
and there is a 
polynomial time $O(\log^2{n})$ approximation algorithm
\cite{DBLP:journals/jal/BodlaenderGHK95}, where $n$ is the number
of vertices. 

The paper \cite{ChTjda2011} studied the relationship between the two
graph chromatic numbers, $\chiump(G)$ and $\chicfp(G)$,
showing that for every graph $G$,
$\chiump(G) \leq 2^{\chicfp(G)}-1$, and providing a sequence of graphs for
which the ratio $\chiump(G)/\chicfp(G)$ tends to $2$.
Moreover, the authors of \cite{ChTjda2011} proved that even checking whether a
given coloring of a graph is conflict-free is coNP-complete
(whereas the same problem for unique-maximum colorings is in P).

Odd colorings with respect to paths of graphs have been recently
studied in \cite{parityvc2011,parityvctrees2012}, independently
from our work. In these papers, they are called \emph{parity
vertex colorings}.

\subsection*{Our results}
In this work, we study the relationship between unique-maximum and
conflict-free colorings. 

First, we give an exact 
answer to the question ``How much larger than $\chicf(H)$ can
$\chium(H)$ be?'' for a general hypergraph $H$.
In section~\ref{sec:hypergraphcfum}, we show that if for a
hypergraph $H$, $\chicf(H) = k > 1$, then $\chium(H)$ is bounded
from above, roughly, by $\frac{k-1}{k} {\card{V}}$,
and this is tight; the result remains true even if we restrict
ourselves to uniform hypergraphs.

Then, we turn to hypergraphs induced by paths in tree graphs and
prove an upper bound for $\chiump(T)$ that is polynomial in 
$\chicfp(T)$, where $T$ is a tree graph. 
We study trees because for general graphs
the only known upper bound for $\chiump(G)$ is 
exponential in $\chicfp(G)$; see \cite{ChTjda2011}.
In section~\ref{sec:treecfum}, we show that for every tree graph
$T$,
$\chiump(T) \leq ({\chicfp(T)})^{3}$ and provide a 
sequence of trees for 
which the ratio $\chiump(T)/\chicfp(T)$ tends at least to 
$\log_{2}{3} \approx 1.58$ (corollary~\ref{last}).
Our results on trees have also implications for the relationship of the
ODD and UM chromatic number of trees.
In particular, corollary~\ref{last} disproves the following 
conjecture from \cite{parityvc2011}: 
``For any tree $T$ we have $\chiump(T) - \chioddp(T) \leq 1$''.
This conjecture was also disproved independently in
\cite{parityvctrees2012}, but our disproof is stronger in the
following sense: in \cite{parityvctrees2012}, the authors give a
sequence of trees for which the ratio $\chiump(T)/\chioddp(T)$
is at least $1.5$, whereas our corollary~\ref{last} implies 
a sequence of trees for which
the aforementioned ratio tends to at least 1.58.
We also improve the trivial lower bound on the ODD chromatic number of 
binomial trees given in \cite{parityvctrees2012} 
(see remark~\ref{rem:binomialtrees}).

Conclusions and open
problems are presented in section~\ref{sec:conclusion}.

\subsection{Preliminaries}

%
\begin{obs}\label{obs:monotonicitysubgraphs}
Each of the graph chromatic numbers
$\chiump$, $\chicfp$, and $\chioddp$, is monotone with respect to
subgraphs, i.e.,
if $H \issubgraphof G$, then $\chigenp(H) \leq \chigenp(G)$, 
where 
$\gensymbol \in \{\text{\textup{um}},\text{\textup{cf}},
  \text{\textup{odd}}\}$.
\end{obs}
\begin{proof}
A subgraph $H$ of a graph $G$ contains a subset of the paths
of $G$. 
\end{proof}

\begin{defi}[Parity vector]
Given a coloring $C \colon V \to \{1,\dots,k\}$ and a set 
$e \issubsetof V$,
the \emph{parity vector} of $e$ 
is an element of $\{0,1\}^k$
in which the $i^{\text{th}}$ coordinate equals 
the parity (0 or 1) of the number of 
elements in $e$ colored with $i$.
\end{defi}

\begin{remark}
A coloring of a hypergraph 
is odd if and only if the parity vector of 
every hyperedge is not the all-zero vector.
\end{remark}


\section{General hypergraphs}
\label{sec:hypergraphcfum}

In general, it is not possible to bound $\chicf$ with a function of
$\chiodd$. For example, the hypergraph $H'$ with
hyperedge set consisting of all triples of
$\{1,\dots,n\}$ has $\chiodd(H')=1$ and 
$\chicf(H')= \left\lceil \fracol{n}{2} \right\rceil$. 
Although $\chicf(H)=1$ implies $\chium(H)=1$,
we can have a big gap as is shown by the following theorem.

\begin{theorem} \label{genhyp}
For a hypergraph $H$
on $n$ vertices, $\chium(H)\le n-\lceil n/\chicf(H)\rceil+1$.
Moreover, this is the best possible bound, i.e., 
for any positive integer $n$ 
there exists a hypergraph on $n$ vertices for which equality holds.
\end{theorem}
\begin{proof}
A simple algorithm achieving the upper bound is the following.
Given a hypergraph $H$ with $\chicf(H)=k$, take a conflict-free coloring of
$H$ with $k$ colors, color the largest color class with color $1$,
all the other vertices with all different colors (bigger than
$1$). It is not difficult to see that this is 
a unique-maximum coloring and that 
it uses at most $n-\lceil n/k\rceil+1$ colors.

For a given $n$ and $k$ equality holds for the hypergraph $H$ 
whose $n$ vertices are partitioned into $k$ (almost) equal parts, 
all of size $\ceil{n/k}$ or $\floor{n/k}$, 
and its edges are all sets of size $2$ and $3$ covering vertices 
from exactly $2$ parts. 
We have $\chicf(H)=k$ because in any conflict-free coloring of $H$ 
there are no two vertices in different parts having the same color and 
$\chium(H)\ge n-\lceil n/k\rceil+1$ because in any unique-maximum coloring 
of $H$ all vertices must have different colors except that the vertices 
of one part can be all colored with $1$. 
\end{proof}
%
%
%



For uniform hypergraphs without small hyperedges, 
we can make the inequality tighter.

\begin{theorem} \label{lreghyp} 
If $l\ge3$ and $k\ge 2$ then for an arbitrary $l$-uniform hypergraph $H$ with $\chicf(H)=k$ having $n\ge 2kl$ vertices we have $\chium(H)\le n-\lceil n/k\rceil-l+4$. 
Moreover, this is the best possible bound, i.e., 
for arbitrary $n\ge 2kl$ there exists a hypergraph for which equality holds.
\end{theorem}


\begin{proof}
The proof is similar to the proof of theorem~\ref{genhyp}, 
although more complicated.

\smallskip

For the first part of the theorem, 
we describe an algorithm that produces a unique-maximum coloring with 
$n-\lceil n/k\rceil-l+4$ colors. 
Given an $l$-uniform hypergraph $H = (V,E)$ on $n$ vertices with $\chicf(H)=k$, 
take a conflict-free coloring $C_{\text{cf}}$ of $H$ using $k$ colors. 
Consider the $k$ color classes $X_i = C_{\text{cf}}^{-1}(i)$,
	for $i \in \{1, \dots, k\}$, 
and without loss of generality assume they are in order 
of non-increasing size, i.e., 
\[\card{X_1} \geq \card{X_2} \geq \dots \geq \card{X_k}.\]

Now, consider the following coloring $C_{\text{um}}$:
Color the vertices of $X_1$ with color $1$, 
color $\min(l-2,|X_2|)$ vertices of $X_2$ with color $2$, 
and color all other vertices with all different colors. 
Observe, first of all, that $C_{\text{um}}$ is a conflict-free
coloring, because it is a refinement of conflict-free coloring
$C_{\text{cf}}$.
We additionally prove that $C_{\text{um}}$ 
is a unique-maximum coloring.
Indeed, for an arbitrary edge $e$, if the maximum color occurring
in $e$ is greater than 2, then 
$e$ has the unique-maximum property, because each color greater than 2
occurs in exactly one vertex of the hypergraph.
Otherwise, the only colors that occur in $e$ are 1 and 2. 
Since
$\card{e} = l$ and color 2 occurs in at most $l-2$ vertices
of $e$, color 1 occurs in at least two vertices of $e$.
But then, $e$ has the conflict-free property if and only if
exactly one vertex of $e$ is colored with 2, i.e., 
$e$ has the unique-maximum property.

The number of colors used in $C_{\text{um}}$ 
is $2+n-\card{X_1}-\min(l-2,\card{X_2})$. 
%
If $\card{X_2} \ge l-2$ then 
(also because $\card{X_1} \ge \lceil n/k \rceil$) this number is at most
$n-\lceil n/k\rceil-l+4$. 
Otherwise, $\card{X_2} < l - 2$ and in that case
the number of colors used is 
$2+n-\card{X_1}-\card{X_2}$.
Using inequalities 
$\card{X_1} \ge \lceil n/k \rceil$
and $\card{X_1} + (k-1)\card{X_2} \ge n$, we get
\[
\card{X_1}+\card{X_2} 
 \ge 
   \frac{1}{k-1} \Bigl( n +(k-2) \Bigl\lceil \frac{n}{k} \Bigr\rceil \Bigr)
 = 
   \Bigl\lceil \frac{n}{k} \Bigr\rceil +
   \frac{1}{k-1} \Bigl( n - \Bigl\lceil \frac{n}{k} \Bigr\rceil \Bigr) 
  .
\]
Then, using inequality 
$\lceil \frac{n}{k} \rceil < \frac{n}{k} +1$,
we get
\[ 
\card{X_1}+\card{X_2} 
 > 
 \Bigl\lceil \frac{n}{k} \Bigr\rceil +
 \frac{1}{k-1} \Bigl( n - \frac{n}{k} - 1 \Bigr)
 =
 \Bigl\lceil \frac{n}{k} \Bigr\rceil +
 \frac{n}{k} - \frac{1}{k-1} 
 \ge 
 \Bigl\lceil \frac{n}{k} \Bigr\rceil +
 \frac{n}{k} - 1 
 .
\]
Finally, using inequality $n \ge 2kl$, we get
\[ 
\card{X_1}+\card{X_2} 
 \geq 
 \Bigl\lceil \frac{n}{k} \Bigr\rceil +
 2l -  1 
 \geq 
 \Bigl\lceil \frac{n}{k} \Bigr\rceil +
 l -  2
 .
\]
Thus, the number of colors used is 
$2+n-\card{X_1}-\card{X_2}
 \le n - \lceil n/k \rceil - l + 4
$.

\medskip

For the second part of the theorem,
given $k$, $l$, and $n$ with $n\ge 2kl$, we construct 
an $l$-uniform hypergraph $H$ 
with $\chicf(H)=k$ and 
$\chium(H)=n-\lceil n/k \rceil -l +4$.
We have $n$ vertices partitioned into $k$ almost equal parts 
$V_1,V_2,\ldots, V_k$, the first $k'$ having size $\ceil{n/k}$, 
the rest having size $\ceil{n/k}-1$. 
The hyperedge set of $H$ consists of all hyperedges of size $l$ for which 
there is a part $V_i$ that intersects the edge in exactly one vertex. 
During the rest of the proof we will use several times the 
pigeonhole principle on the above defined parts.

It is not difficult to see that 
the coloring defined by the partition $V_1$, \ldots, $V_k$ 
is a conflict-free coloring using $k$ colors, 
i.e., $\chicf(H) \leq k$.
We now prove that 
there is no conflict-free coloring using less than $k$ colors. 
For that, take a conflict-free coloring $C_{\text{cf}}$ of $H$ with the 
optimal number of colors. 

For a color $c$, if its color class $C_{\text{cf}}^{-1}(c)$ 
is covered by some $V_i$ then its size is at most $|V_i|$. 
Thus, we have at most $k'$ such color classes of size 
$\lceil n/k\rceil$ and the rest is of size at most $\lceil n/k\rceil-1$.
For a color $c$ for which its color class $C_{\text{cf}}^{-1}(c)$ 
is not covered by one part of the partition, $C_{\text{cf}}^{-1}(c)$ 
must intersect at least two different parts, 
$V_i$ and $V_j$ with $i \ne j$. 
If $|C^{-1}(c)|> 2l-4$ then either there is an $l$-subset of 
$C^{-1}(c)$ having exactly one point from $V_i$ or there is an $l$-subset of 
$C^{-1}(c)$ having exactly one point from $V_j$, 
which is a contradiction as these $l$-subsets 
would be monochromatic edges of $H$. 
If $|C_{\text{cf}}^{-1}(c)|\le 2l-4$, 
then our assumption $n\ge 2kl$ implies 
$2l-4<\lceil n/k\rceil-1$, i.e., 
if a color class is not covered by some $V_i$ 
then it is smaller than $\lceil n/k\rceil-1$. 
Hence, the only way we can color all the vertices using only 
$k$ colors is by not having color classes intersecting two parts and 
by having every color class equal to a part of the partition. 
Thus, we proved that if $n\ge 2kl$ then $\chicf(H)=k$ 
and also that the only optimal coloring is the one defined by the partition.

Now, take a unique-maximum coloring 
$C_{\text{um}}$
with the optimal number of colors. 
We prove that it uses at least $n-\ceil{n/k}-l+4$ colors. 
We define $c$ to be the \emph{biggest color} for which 
there are at least $2$ vertices having color $c$.
By definition every color bigger than $c$ is used only at most once in this 
unique-maximum coloring.
We define $Y$ to be the set of vertices with color $c$ and $Y'$ to
be the set of vertices with color $c$ or smaller. 

\begin{obs} \label{obsunique}
Coloring  $C_{\text{um}}$  uses $n-|Y'|+c$ colors.
\end{obs}

Since every edge has the unique-maximum property, the following is true.

\begin{obs} \label{obscc}
There is no edge that contains only vertices from $Y'$ 
and contains at least two vertices from $Y$.
\end{obs}

If $Y'$ can be covered by some $V_i$, then $|Y'|\le\lceil n/k\rceil$ and so we used at least $n-\lceil n/k\rceil+1\ge n-\lceil n/k\rceil-l+4$ colors altogether.
If $Y'$ cannot be covered by one $V_i$ of the partition, then we have 3 cases:
\begin{itemize}

\item[(i)] $Y$ cannot be covered by one part.

In this case there are two vertices with color $c$ that are in
different parts, $x$ in $V_i$ and $y$ in $V_j$ with $i \ne j$. 
If $|Y'|>2l-4$, then there is an $l$-subset of $Y'$ containing only 
$x$ from $V_i$ and $l-1$ vertices from other parts (including $y$)
or an $l$-subset of $Y'$ containing only $y$ from $V_j$ and $l-1$ vertices from other parts (including $x$). 
Any of these two subsets would be an edge of $H$ contradicting 
observation~\ref{obscc}. 
Thus, $|Y'|\le 2l-4$ and so we used at least 
$n-(2l-4)+1\ge n-\lceil n/k\rceil-l+4$ colors 
(for the last inequality we used that $n\ge 2kl$).
\end{itemize}

In the next two cases, as $Y$ is contained in some $V_i$, 
but $Y'$ is not, we have $Y'\ne Y$ and thus $c\ge 2$.
\begin{itemize}
\item[(ii)] $Y$ is contained in some $V_i$ and 
            $Y'$ can be covered by two parts $V_i$ and $V_j$.

If $|Y'\cap V_i|> l-2$ then there exists an $l$-subset of $Y'$ containing exactly one vertex from $V_j$ and $l-1$ vertices from $V_i$ such that at least two of these vertices have color $c$. 
This subset would be an edge of $H$ contradicting observation~\ref{obscc}. 
Thus, $|Y'|\le |V_j|+l-2\le \lceil n/k\rceil+l-2$, 
and as $c\ge 2$,
we used at least $n-(\lceil n/k\rceil+l-2)+2=n-\lceil n/k\rceil-l+4$ colors.

\item[(iii)] $Y$ is contained in some $V_i$ and 
             $Y'$ cannot be covered by two parts.

In this case $Y'$ contains points from at least three parts, 
$V_i$, $V_j$ and some $V_{h}$. 
Now, it is easy to see that if $|Y'|>2l-6$ 
then there exists 
an $l$-subset of $Y'$ containing at least two vertices from $V_i$ with color $c$ and
that either has exactly one vertex from $V_j$ 
or exactly one vertex from $V_{h}$. 
This subset would be an edge in $H$ contradicting observation~\ref{obscc}. 
Thus, $|Y'|\le 2l-6$, and as $c\ge 2$, 
we used at least 
$n-(2l-6)+2\ge n-\lceil n/k\rceil-l+4$ colors. 
\qedhere
\end{itemize}
\end{proof}

\section{Tree graphs}%
\label{sec:treecfum}
In this section, to ease readability 
we use $\UM$ for $\chiump$, $\CF$ for $\chicfp$ and $\ODD$ for $\chioddp$.
We denote by $P_n$ the path graph with $n$ vertices.
As a warm-up we prove a simple claim about the odd chromatic 
number of the path graph. 
Our proof is a showcase of a parity vector argument,
which we are going to also use later.
For completeness, we include a computation of the conflict-free and 
unique-maximum chromatic numbers of the path graph
\cite{ELRS03jo}.
(Throughout this paper we use base 2 logarithms, which are denoted
by ``$\log$''.)%

\begin{clai}\label{path}
For $n \geq 1$, $\ODD(P_n)=\CF(P_n)=\UM(P_n)=
  \left\lceil \log (n+1)\right\rceil$.
\end{clai}
\begin{proof}
It is easy to see that $\UM(P_n)\le \left\lceil \log (n+1)\right\rceil$: 
assign the biggest color only to a median vertex of the path
and then use recursion. 
%
Since we know that $\UM(P_n)\ge \CF(P_n)\ge \ODD(P_n)$, 
it is enough to prove that $2^{\ODD(P_n)}> n$. 
Take the $n$ paths starting from one endpoint. 
If there were two with the same parity vector, 
their symmetric difference (which is also a path) 
would contain an even number of each color. 
Thus, we have at least $n$ different parity vectors, 
none of which is the all-zero vector. 
But the number of non-zero parity vectors is at most $2^{\ODD(P_n)}-1$. 
\end{proof}

\subsection{Upper bound for unique-maximum number of binary trees}

We denote by $B_d$ the (rooted) complete binary tree with $d$ 
levels (and $2^{d}-1$ vertices).
By convention, $B_0$ is the empty graph.
It is easy to see that $\UM(B_d)=d$; for an optimal unique-maximum
coloring, color the leaves of $B_d$ with color 1, their parents
with color 2, and so on, until you color the root with color $d$;
for a matching lower bound, use induction on $d$.
In this section, we prove an upper bound for $\UM(B_d)$ that is
quadratic on $\CF(B_d)$. In fact, we will prove a stronger
statement, that is, a bound for $\UM(B_d)$ that is quadratic on
$\ODD(B_d)$. 
Moreover, 
instead of proving a bound just for complete binary trees, we
are going to prove a bound for subdivisions of 
complete binary trees, because we will need that later in
subsection~\ref{subsec:arbitrarytrees}. 
We first need the following definitions.

\begin{defi}
A graph $H$ is a  \emph{subdivision} of $G$ 
if $H$ is obtained by
substituting every edge $uv$ of $G$ by a path of new internal 
vertices between $u$ and $v$. 
The original vertices of
$G$ in $H$ are called \emph{branch vertices}.
\end{defi}

\begin{defi}
Given is a rooted tree $T$ and a rooted subtree $T'$ of $T$.
We say that $T'$ is \emph{compatible with} $T$ if any two 
vertices of $T'$ have the same ancestor-descendant relation
in both $T'$ and $T$.
\end{defi}

We are now ready to state the following useful lemma.

\begin{lemma}\label{lemma:vector}%
Let $B^*$ be a subdivision of $B_d$.
Suppose we color (without any restrictions) 
the vertices of $B^*$ with $k$ colors.
Then,
there exists a vector $a=(a_1,a_2, \dots ,a_k)$ 
such that $\sum_{i=1}^{k} a_i \geq d$ and for every $i \in \{1,\dots,k\}$, 
$B^*$ contains a subdivision $T^i$ of $B_{a_i}$ such that 
(1)~$T^i$ is compatible with $B^*$ and
(2)~the branch vertices of $T^i$ are all colored with $i$.
\end{lemma}
\begin{proof}
We construct the vector $a$ by induction on $d$. 

For $d=1$, $B^*$ has exactly one vertex, say with color $i$. 
Then, $a$ is such that $a_i=1$  and every other coordinate is 0.

For $d > 1$, consider the tree $B^*$ rooted at the branch vertex
$v$ that corresponds to the root of $B_d$.
Each of the left and right subtrees of $v$ contains a subdivision
of $B_{d-1}$. Thus, by the inductive hypothesis, we construct
vector $a'$ for the left subtree and $a''$ for the right subtree.
If $a' \neq a''$, then $a$ is such that 
$a_i = max(a'_i, a''_i)$, for $i \in \{1,\dots,k\}$.
If $a' = a''$, then $a$ is such that $a_i = a'_i + 1$
for the color $i$ of the root $v$ and $a_j = a'_j$ for $j \neq i$. 
\end{proof}

\begin{theorem}\label{binary}
For $d \geq 1$ and for every subdivision $B^*$ of $B_d$, 
$\ODD(B^*) \geq \sqrt{d}$. 
%
\end{theorem}
\begin{proof}
Fix an optimal odd coloring with $k$ colors. 
Fix a color $i \in \{1,\dots,k\}$ 
for which in lemma~\ref{lemma:vector} we have $a_i \geq d/k$. 

Consider the $2^{a_i-1}$ paths that originate in a leaf of the 
$B_{a_i}$ subdivision and end in its root branch vertex.
%
We claim that 
the parity vectors of the $2^{a_i-1}$ paths must be all different.
Indeed, if there were two paths with the same parity vector,
then the symmetric difference of the paths plus their lowest
common vertex would form a path where the parity of each color is
even, except maybe for color $i$, but since this new path starts and ends with color $i$, deleting any of its ends yields a path whose parity vector 
is the all-zero vector, a contradiction.

There are at most $2^k-1$ parity vectors, 
thus $2^k-1\ge 2^{a_i-1}\ge  2^{\left\lceil d/k\right\rceil-1}$. 
From this we get $k> \left\lceil \fracol{d}{k} \right\rceil-1$ 
which is equivalent to $k\ge\left\lceil \fracol{d}{k} \right\rceil$ 
using the integrality.  
Thus, $k\ge\sqrt d$. 
\end{proof}

\begin{remark}\label{rem:binomialtrees}
The (rooted) \emph{binomial tree} $T_d$ with $2^{d-1}$ vertices is defined
as follows: $T_1$ is a single vertex; for $d>1$, $T_d$ consists of
two disjoint copies of $T_{d-1}$ and an edge between their two roots,
whereas the root of $T_d$ is the root of the first copy.
These trees are used in \cite{parityvc2011,parityvctrees2012}
and play a similar role to binary trees in our work.
It is not difficult to prove by induction that 
$T_{2d-1} \issupergraphof B_{d}$.
As a result, theorem~\ref{binary} implies 
$\ODD(T_d) = \Omega(\sqrt{d})$, 
which improves the trivial lower bound 
$\ODD(T_d) = \Omega(\log{d})$ from \cite{parityvctrees2012}.
\end{remark}

\subsection{Upper bound for unique-maximum number of arbitrary trees}%
\label{subsec:arbitrarytrees}

We will try to find either a long path or a subdivision of 
a deep complete binary tree in
every tree with high UM chromatic number. 
For this, we need the notion of UM-critical trees and their
characterization from \cite{orderedcoloring}.

\begin{defi}
A graph is \emph{UM-critical}, if the UM chromatic number 
of any of its subgraphs is smaller than its UM chromatic number.
We also say that a graph is \emph{$k$-UM-critical}, if it is
UM-critical and its UM chromatic number equals $k$.
\end{defi}

\begin{example}
The complete graph $K_k$ and the path with $2^{k-1}$ vertices 
are both $k$-UM-critical.
For $k\le 3$ there is a unique $k$-UM-critical tree, the path with $2^{k-1}$ vertices.
Consider the following tree $T$ on $8$ vertices:
Take two copies of $P_4$ and draw an edge from an internal vertex of
one $P_4$ to an internal vertex of the other $P_4$.
Tree $T$ is
$4$-UM-critical
and $\CF(T)=3$. ($T$ is the smallest tree where the CF and UM 
chromatic numbers differ.)
\end{example}

\begin{theorem}[Theorem 2.1 in \cite{orderedcoloring}]\label{umcrit}
For $k > 1$, 
a tree is $k$-UM-critical if and only if 
it has an edge that connects two $(k-1)$-UM-critical trees.
\end{theorem}


\begin{remark}
A $k$-UM-critical tree has exactly $2^{k-1}$ vertices and
the connecting edge must always be the central edge of the tree.
This implies that there is a unique way to partition the vertices 
of the $k$-UM-critical tree to two sets of vertices, each inducing 
a $(k-1)$-UM-critical tree, and so on.
\end{remark}

Now we can define the \emph{structure trees} of UM-critical trees.

\begin{defi}
For $l \in \{0,\dots,k-1\}$,
the \emph{$l$-deep structure tree} of a $k$-UM-critical tree is the 
tree graph with a vertex for every one of the $2^l$ $(k-l)$-UM-critical 
subtrees that we obtain by repeatedly applying theorem
\ref{umcrit}, and an edge between two vertices 
if the corresponding $(k-l)$-UM-critical  
subtrees have an edge between them in the  $k$-UM-critical tree.
\end{defi}

\begin{example}
The $0$-deep structure tree of any UM-critical tree is a vertex.
The $1$-deep structure tree of any UM-critical tree is an edge.
The $2$-deep structure tree of any UM-critical tree is a path with $4$ vertices.
The $(k-1)$-deep structure tree of a $k$-UM-critical tree is itself.
\end{example}

\begin{remark}
It is not difficult to prove that 
the $l$-deep structure tree of a 
UM-critical tree is an $(l+1)$-UM-critical tree.
\end{remark}

We start with a few simple observations.

\begin{obs}\label{UMpath}
If an $(l+1)$-UM-critical tree has no vertex of degree at least $3$, then it is the path with $2^l$ vertices.
\end{obs}
\begin{proof}
Delete the central edge and use induction. 
\end{proof}

\begin{obs}\label{longpath}
If an $(l+2)$-UM-critical tree has only one vertex of degree at least $3$, then it contains a path with $2^l$ vertices that ends in this vertex.
\end{obs}
\begin{proof}
After deleting its central edge, one of the resulting $(l+1)$-UM-critical trees must be a path that was connected to the rest of the graph with one of its ends, thus we can extend it until the high degree vertex. 
\end{proof}

\begin{obs}
If a tree contains two non-adjacent vertices with degree at least $3$, then it 
contains a subdivision of $B_3$.
\end{obs}
\begin{proof}
The non-adjacent degree $3$ vertices will be the second level of the binary tree, and any vertex on the path connecting them the root. 
\end{proof}

\begin{clai}
\label{claim:level3}%
An $(l+2)$-UM-critical tree contains a path with $2^l$ vertices or a 
subdivision of $B_3$.
\end{clai}
\begin{proof}
Because of the previous observations, we can suppose that our tree
has exactly two vertices with degree at least $3$ and these are
adjacent. If the central edge is not the one between these
vertices, then the graph must contain an $(l+1)$-UM-critical
subgraph without any vertex with degree at least $3$, thus it is
the path with $2^l$ vertices because of observation~\ref{UMpath}.
If it connects the two high degree vertices, then, 
using observation~\ref{longpath}, we have two paths with $2^{l-1}$ vertices in the $(l+1)$-UM-critical subgraphs obtained by deleting the central edge ending in these vertices, thus with the central edge they form a path with $2^l$ vertices. 
\end{proof}

We are now ready to prove our main lemma, before the proof of the
upper bound.

\begin{lemma}\label{lemma:pathorbinarytree}
For $k\ge 3$ and any $l$, every $k$-UM-critical tree contains a path with $2^l$ vertices or a 
subdivision of $B_{\left\lceil \frac{k+l+3}{l+2}\right\rceil}$.
\end{lemma}
\begin{proof}
The proof is by induction on $k$.
For $3\le k\le l+1$, the statement is true 
since $B_2=P_3$.
For $l+2\le k\le 2l+3$, the statement is equivalent to our
claim~\ref{claim:level3}. 
For $k>2l+3$, take the $(l+2)$-deep structure tree $S$ of the tree.
If $S$ does not contain a path with $2^l$ vertices, then, 
using  claim~\ref{claim:level3},
$S$ contains a subdivision of $B_3$.
Every one of the four leaf branch vertices of the above $B_3$ subdivision 
corresponds to a $(k-l-2)$-UM-critical subtree of the original tree.
By induction,  each one of the above four subtrees 
must contain a path with $2^l$ vertices or a
subdivision of the complete binary tree with 
$\left\lceil \frac{k-l-2+l+3}{l+2}\right\rceil = \left\lceil \frac{k+l+3}{l+2}\right\rceil-1$ levels.
If any of them contains the path, we are done.
If each one of them contains a $B_{\left\lceil \frac{k+l+3}{l+2}\right\rceil-1}$
subdivision, then for every one of the four leaves, we can connect 
at least one of the two disjoint 
$B_{\left\lceil \frac{k+l+3}{l+2}\right\rceil-2}$ subdivisions of the 
$B_{\left\lceil \frac{k+l+3}{l+2}\right\rceil-1}$
subdivision in the leaf (as in figure~\ref{fig:bintree}, where each of the four relevant
$B_{\frac{k+l+3}{l+2}-2}$ subdivisions and the paths connecting them are shown with heavier lines) to obtain a subdivision of a complete binary tree
with $\left\lceil \frac {k+l+3}{l+2}\right\rceil-2+2$ levels, 
thus we are done. 
\end{proof}

\begin{figure}[ht]
\begin{center}
\scalebox{0.58}{\includegraphics{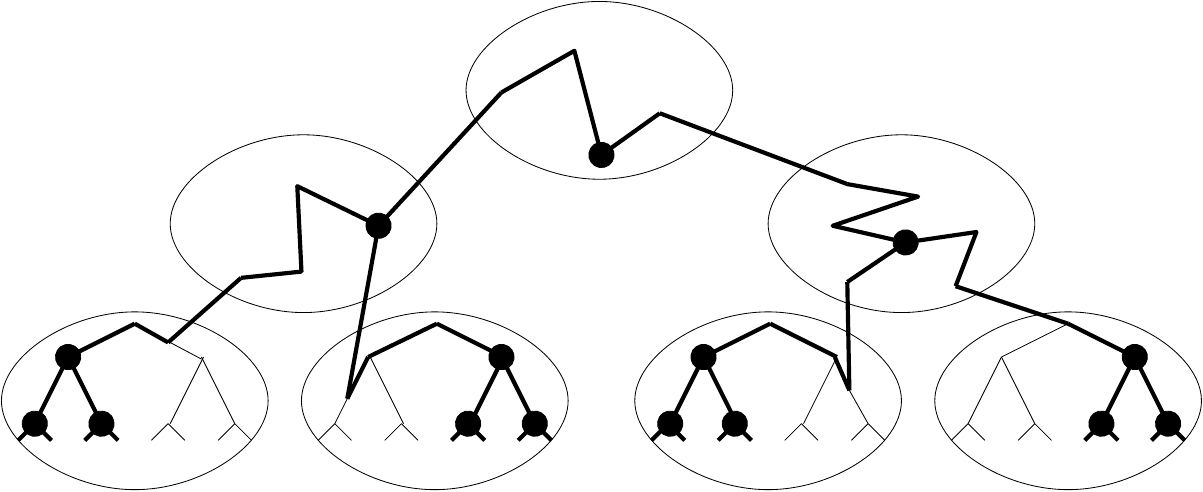}}
\caption{Constructing a deep binary tree using induction for structure trees}
\label{fig:bintree}
\end{center}
\end{figure}

\begin{theorem}
For every tree $T$, 
$\ODD(T)\ge (\UM(T))^{\frac{1}{3}}-O(1)$. 
\end{theorem}
\begin{proof}
If $\UM(T)=k$, then $T$ contains a $k$-UM-critical tree, which
(according to lemma~\ref{lemma:pathorbinarytree}) contains a
$P_{2^l}$ or a subdivision $B^*$ of 
$B_{\left\lceil \frac{k+l+3}{l+2}\right\rceil}$.
Using monotonicity of $\ODD$ with respect to subgraphs
(observation~\ref{obs:monotonicitysubgraphs}),
together with 
$\ODD(P_{2^l}) = l+1$ (claim~\ref{path}) and 
$\ODD(B^*) \geq \sqrt \frac {k+l+3}{l+2}$
(from theorem~\ref{binary}),
we get $\ODD(T) \geq \max \Bigl(l+1, \sqrt{\frac{k+l+3}{l+2}} \Bigr)$.
%
%
Choosing $l$ to be the closest integer 
to the solution of $l+1 = \sqrt{\frac {k+l+3}{l+2}}$, 
we get $l = k^\frac{1}{3} + \Theta(1)$.
Therefore,  $\ODD(T) \geq (\UM(T))^\frac 13-O(1)$. 
\end{proof}

\subsection{Trees with different unique-maximum and conflict-free numbers}

We have seen that $\UM(B_d)=d$. We intend to show conflict-free colorings of some complete binary trees that use substantially less colors.
We start with a simple example demonstrating our method. 

\begin{clai}
$\CF(B_7) \leq 6$.
\end{clai}
\begin{proof}
Color the root with $1$, the second level with $2$. Deleting the
colored vertices leaves four $B_5$ subtrees. In each of these
subtrees, every level will be monochromatic. From top to bottom,
in the first use the colors $3,4,5,1,2$, in the second
$4,5,6,1,2$, in the third $5,6,3,1,2$ and in the forth
$6,3,4,1,2$. It is not difficult to verify that this is indeed a conflict-free coloring (but it will also follow from later results). Observe that in the top 2 levels 2 colors are used, in the next 3 levels 4 colors, and in the last 2 levels the same 2 colors are used as the ones in the top level.
\end{proof}

\begin{corollary}\label{iterate}
$\CF(B_{2(r+1)+3r}) \leq 4r + 2$.
\end{corollary}
\begin{proof}
In the previous construction, every leaf had color $2$ and their
parents had color $1$. Every such three vertex part can be the top
of a new tree, similar to the original, and replacing $3,4,5,6$
with four new colors. This gives a tree with $12$ levels and $10$
colors. It is not difficult to verify that this is indeed a conflict-free coloring (but it will also follow from later results). Repeatedly applying this procedure, so that we have colors 1, 2 appearing in $2(r+1)$ levels and $r$ disjoint sets of 4 colors each, 
we get a coloring of $B_{2(r+1)+3r}$ using $4r+2$ colors.
\end{proof}

To examine more closely why these colorings are conflict-free, 
we need to define some notions.

\begin{defi}
An \emph{ordered set} is a sequence in which no element repeats.
Two ordered sets are \emph{equal as sets} if they have the same
elements (ignoring the order of elements).
\end{defi}

\begin{defi}
A family $\calF$ of ordered sets is said to be \emph{prefix set-free}, 
if any prefix of any ordered set in $\calF$ is different from any other 
ordered set in $\calF$ as a set (without the ordering).
If the ground set has $n$ elements, 
every sequence has length at least $k$, 
and the cardinality of $\calF$ is at least $2^d$, 
then we say that $\calF$ is a $[k,d,n]$ \emph{PSF family}.
\end{defi}

\begin{example}
$\{\oset{1,3},\oset{1,2,3}\}$ is a $[2,1,3]$ PSF family
and 
$\{\oset{1},\oset{2,1},\allowbreak \oset{2,3}, \allowbreak \oset{3,1},
   \allowbreak \oset{3,1,2}\}$ is a $[1,2,3]$ PSF family but 
$\{\oset{2,1},\oset{1,2,3}\}$ is not a PSF family. 
\end{example}

\begin{clai}
For any $[k,d,n]$ PSF family $d\le \log \sum_{i=k}^n \binom{n}{i}$.
\end{clai}
\begin{proof}
Any two ordered sets in the PSF family must differ as sets.
\end{proof}

\begin{clai}\label{bigPSF}
There is a $[k,d,n]$ PSF family with 
$d= \left\lfloor \log \binom{n}{k} \right\rfloor$.
\end{clai}
\begin{proof}
Take all $k$-element subsets of $\{1,\dots,n\}$ and order each arbitrarily.
\end{proof}

Since these bounds do not differ much if 
$k>(\frac{1}{2} +\epsilon)n$, we do not attempt to get sharper bounds.

\begin{theorem}
If there is a $[k,d,n]$ PSF family where the size of every set is at most $k+d$, 
then $\CF(B_{d(r+1)+kr}) \leq nr+d$.
\end{theorem}
\begin{proof}
First, we show that $\CF(B_{k+2d}) \leq n+d$.
Color the top $d$ levels with $d$ colors.
%
Remove the colored vertices and consider the $2^d$ $B_{k+d}$
subtrees left. To each associate an ordered set from the $[k,d,n]$
PSF family and color the whole $i^{\text{th}}$ level with one color, 
the $i^{\text{th}}$ element of the associated ordered set.  Deleting also these colored vertices, we are left with subtrees with at most $d$ levels, 
which we can color with (at most) the same $d$ colors 
we used for the top levels. It is not difficult to check that
the above procedure produces a conflict-free coloring.
By repeating the above procedure $r$ times for $B_{d(r+1)+kr}$, as in corollary \ref{iterate}, we obtain  $\CF(B_{d(r+1)+kr}) \leq nr+d$.
\end{proof}


\begin{corollary}\label{last}
For the sequence of complete binary trees, $\{B_i\}_{i=1}^{\infty}$,
the limit of the ratio of the UM to the CF chromatic
number is at least $\log 3\approx 1.58$.
\end{corollary}

\begin{proof}
%
%
Since $\CF(B_{d(r+1)+kr}) \leq nr+d$, 
the ratio of UM to CF for $B_{d(r+1)+kr}$ 
is at least $(d(r+1)+kr)/(nr+d)$, which tends to $(d+k)/n$ as $r \to \infty$.
From claim \ref{bigPSF} we can choose $d = \floor{\log{\binom{n}{k}}}$. 
If we substitute $k$ with $xn$, then a short calculation shows that to maximize 
$(d+k)/n$ we have to maximize $x+H(x)$, where $H(x)=-x\log x - (1-x)\log(1-x)$
(entropy). 
The function $x+H(x)$ attains its maximum at $x = 2/3$, 
giving a value of $\log 3$ as a lower bound for the limit. 

For completeness, we also include a proof of the existence of the limit.
For brevity, denote $\CF(B_i)$ by $c_i$.
Our goal is to show that sequence $f$, 
with $f_i = c_i/i$, i.e., the ratio of
$CF$ to the $UM$ for $B_i$ has a finite limit.
If the limit exists, then it is finite, because $f_i \in (0, 1]$.
We know that $c_i$ is monotone increasing.
We also know that $c_{i+1}\le {c_i}+1$ since we can take two copies of a good 
CF-coloring of depth $i$ and join them with a root having a new color.
In fact, we even know 
$c_{i+j}\le c_i+c_j$ because we can take a good CF-coloring of depth $j$ and 
put $2^j$ copies of a good CF-coloring of depth $i$ under each of its leaves.
Then, for $n \geq i$, 
\[
\frac{c_n}{n} \leq \frac{\ceil{n/i} c_i}{n} < 
  \left(\frac{n}{i} +1 \right) \frac{c_i}{n} =
  \frac{c_i}{i} + \frac{c_i}{n} \leq 
  \frac{c_i}{i} + \frac{i}{n},
\]
that is, $f_n \leq f_i + i/n$. With this, it is not difficult to prove,
using standard arguments, 
that no two subsequences of $f$ have different
limits, and thus $f$ has a limit.
%
\end{proof}

\begin{remark}
Since $\ODD(T) \leq \CF(T)$, corollary~\ref{last} is also true for
the ratio of the UM to the ODD chromatic number.
In \cite{parityvctrees2012}, the authors prove that for the
sequence of binomial trees (see remark~\ref{rem:binomialtrees}),
the ratio of the CF to the ODD chromatic number
tends at least to $1.5$, disproving a conjecture of 
\cite{parityvc2011}. Therefore, our result disproves the
aforementioned conjecture in a stronger sense (our limit is at
least $\log{3} \approx 1.58$). 
\end{remark}



\section{Discussion and open problems}\label{sec:conclusion}

In the literature of conflict-free coloring, hypergraphs that are
induced by geometric shapes have been in the focus. It would be
interesting to show possible relations between unique-maximum and conflict-free chromatic
numbers in this setting.

The exact relationship between the two
chromatic numbers with respect to paths for general graphs still
remains an open problem. In
\cite{ChTjda2011}, only graphs which have unique-maximum
chromatic number about twice the conflict-free chromatic number
were exhibited, but the only bound proved on $\chiump(G)$ was
exponential in $\chicfp(G)$. 
In fact it is even possible that $\chiump(G)\le2\chicfp(G)-2$.
The first step to prove this would be to show that 
$\chiump(T)=O(\chicfp(T))$ for trees. 
It would also be interesting to
extend our results to other classes of graphs.

\subsection*{Acknowledgment}
We would like to thank G\'eza T\'oth for fruitful discussions and ideas about improving the lower bound in corollary \ref{last}.

\bibliographystyle{plain}
\bibliography{cfcolor}

\end{document}